\documentclass[11pt,reqno]{amsart}

\usepackage{amssymb}
   \DeclareMathOperator{\e}{e}

\usepackage{hyperref}
\usepackage{graphicx}
\usepackage{hyperref}
\usepackage{url}
\usepackage{amsfonts}

\usepackage{amsthm}
\usepackage{mathrsfs}
\usepackage{latexsym}
\usepackage{amsmath}
\usepackage{setspace}
\usepackage{pstricks}
\usepackage[latin1]{inputenc}
\usepackage[normal]{subfigure}

\def\cprime{$'$}

\newtheorem{theorem}{Theorem}[section]

\newtheorem{definition}[theorem]{Definition}

\newtheorem{corollary}[theorem]{Corollary}
\newtheorem{lemma}[theorem]{Lemma}

\newtheorem{remark}[theorem]{Remark}
\newtheorem{example}[theorem]{Example}

\newcommand{\prts}[1]{\left(#1\right)}
\newcommand{\prtsr}[1]{\left[#1\right]}
\newcommand{\pfrac}[2]{\prts{\dfrac{#1}{#2}}}
\newcommand{\eps}{\varepsilon}
\newcommand{\sgm}{\sigma}

\title{Integral conditions for nonuniform $\mu$-dichotomy on the half-line}

\author{Ant\'{o}nio J. G. Bento}
\address{Ant\'{o}nio J. G. Bento, Departamento de Matem\'{a}tica, Universidade da Beira Interior, 6201-001 Covilh\~{a}, Portugal}
\email{bento@ubi.pt}

\author{Nicolae Lupa}
\address{Nicolae Lupa, Department of Mathematics, "Politehnica" University of Timi\c soara, Victoriei Square 2, 300006 Timi\c soara, Romania}
\email{nicolae.lupa@upt.ro}

\author{Mihail Megan}
\address{Mihail Megan, Academy of Romanian Scientists, Independen\c tei 54, 050094 Bucharest, Romania}
\email{megan@math.uvt.ro}

\author{C\'{e}sar M. Silva}
\address{C\'{e}sar M. Silva, Departamento de Matem\'{a}tica, Universidade da Beira Interior, 6201-001 Covilh\~{a}, Portugal}
\email{csilva@ubi.pt}

\begin{document}

\begin{abstract}
We give necessary integral  conditions and sufficient ones for the existence of a general concept of $\mu$-dichotomy for evolution operators defined on the half-line which includes as particular cases the well-known concepts of nonuniform exponential dichotomy and nonuniform polynomial dichotomy,  and also contains new situations.
Additionally, we consider an adapted notion of Lyapunov function and use our results to obtain necessary and sufficient conditions for the existence of nonuniform $\mu$-dichotomies using these Lyapunov functions.
\end{abstract}

\subjclass[2010]{47D06,34D09}
\keywords{Evolution operators, nonuniform dichotomies, Datko theorem}

\maketitle

\section{Introduction}

The notion of exponential dichotomy is a fundamental tool in the study of stability of difference and differential equations and can be traced back to the work of Perron~\cite{Pe.1930} on the stability of ordinary differential equations, and of Li~\cite{Li.1934} for discrete time systems. We also refer to the book of Chicone and Latushkin
\cite{Ch.La.1999} for  important results in infinite-dimensional spaces and the book of P\"otzsche \cite{Pot.2010} for nonautonomous discrete systems.

In some situations, in particular for nonautonomous systems, the classical concept of (uniform) exponential dichotomy is too restrictive and it is important to look for more general hyperbolic behavior.
We can identify, at least, two ways to generalize this concept: allow some loss of hyperbolicity along the trajectories, a path leading to notions similar to Pesin's nonuniform hyperbolicity~\cite{Pe.1976,Pe.1977-1,Pe.1977-2}, and consider asymptotic behavior that is not necessarily exponential, an approach followed by Naulin and Pinto in \cite{Na.Pi.1995,Pi.1994}, where the authors considered uniform dichotomies with asymptotic behavior given by general growth rates.
In recent years, a large number of papers study different aspects of the dynamical
behavior of systems
with nonuniform exponential dichotomies, a type of dichotomic behavior where some exponential loss of hyperbolicity along the trajectories is allowed (see for example the work of Barreira and Valls~\cite{Ba.Va.2008-1} and papers \cite{Lu.Me.2014,Pr.Pr.Cr.2012,Sa.Ba.Sa.2013}).
Also,  several results  were obtained in \cite{Ba.Me.Po,Ba.Ch.Va.2013,Ba.Va.2008-2,Be.Si.,Be.Si.2013,Ch} for general nonuniform behavior and in \cite{Ba.Va.2009,Be.Si.2009,Be.Si.2012}
for nonuniform polynomial behavior.

One of the most important results in the stability theory of evolution operators is due to Datko \cite{Da.1972} which has given an integral characterization of uniform exponential stability.  This characterization is used to obtain a necessary and sufficient condition for uniform exponential stability in terms of Lyapunov functions. Preda and Megan have extended Datko theorem to uniform exponential dichotomy \cite{Pr.Me.1985}. Generalizations of this result in the case of nonuniform exponential dichotomy are given in \cite{Lu.Me.2014,Me.1996,Pr.Pr.Cr.2012}. For more details and history about Datko theorem we refer the reader to \cite{Sa.2010-1}.

In this paper, we consider a notion of dichotomy which is both nonuniform and not necessarily exponential in the general context of evolution operators defined on the half-line, with the purpose of obtaining necessary conditions and sufficient ones in the spirit of Datko's results. Additionally, we use our Datko type theorem to obtain necessary and sufficient conditions for the existence of nonuniform $\mu$-dichotomies in terms of suitable Lyapunov functions.

We note that the use of Lyapunov functions in the study of (uniform) exponential stability has a long story that goes back to the seminal work of Lyapunov~\cite{Ly.1892}. Corresponding results for exponential dichotomies were first considered by Ma{\u\i}zel{\cprime} \cite{Ma.1954}. We refer the book~\cite{Mi.Sa.Ku.2003} for the relation between (uniform) exponential dichotomies and Lyapunov functions. For stability results via Lyapunov functions in the context of delay equations, we refer the work of Hatvani and collaborators~\cite{Bu.Ha.1989, Bu.Ha.1990, Ha.2002}.

 In the context of nonuniform exponential dichotomies, Megan and Bu\c se discussed in \cite{Me.Bu.1993} the relation between Lyapunov functions and existence of dichotomies. Recently, for a different concept of nonuniform exponential dichotomy, Barreira and Valls~\cite{Ba.Va.2009-1} used quadratic Lyapunov functions, that are Lyapunov functions obtained from quadratic forms,
 to establish the existence of nonuniform exponential dichotomies. In a more recent work~\cite{Ba.Ch.Va.2013}, the authors gave a complete characterization of the existence of a general type of nonuniform dichotomy (it is considered the case of different growth rates for the uniform and nonuniform part of the dichotomy) in terms of suitable Lyapunov functions, but the results are restricted to the case of nonautonomous linear differential equations in finite-dimensional spaces. We point out that none of the results in~\cite{Ba.Ch.Va.2013,Ba.Va.2009-1} and in our paper imply the results in the other.

 We emphasize that the dichotomies considered in this paper include as particular cases the  concepts of nonuniform exponential dichotomy and nonuniform polynomial dichotomy. We show that our results extend previous theorems  and also contain new situations. We remark that, by allowing growth rates that are not exponential, we are considering situations where the classical Lyapunov exponents can be zero. Moreover, note that we do not need to assume the invertibility of the evolution operators on the whole space, which allow us to apply our results to compact operators defined in infinite-dimensional spaces.

\section{Notions and preliminaries}

Let $X$ be a Banach space and let $\mathcal{B}(X)$ be the Banach algebra of all
bounded linear operators on $X$.
Throughout this paper $\mathbb{R}^{+}_{0}$  denotes the set of non-negative real numbers and $\mathbb{N}^*$ is assumed to be the set of positive integers. We also consider  $\Delta$  the set defined by
$$\Delta=\left\{(t,s) \in \mathbb{R}^{+}_{0} \times \mathbb{R}^{+}_{0} \colon t \geq s\right\}.$$

We first recall the definition of the evolution operators:
\begin{definition}\label{d.ev} \rm
An operator valued function  $U:\Delta\to\mathcal{B}(X)$ is said to be  an  \emph{evolution operator}  if
\begin{enumerate}
\item $U(t,t)=\mathrm{Id}$  for every $t\geq 0$;

\item $U(t,\tau)U(\tau,s)=U(t,s)$ for all $ t\geq \tau\geq s\geq 0$;

\item $(t,s)\mapsto U(t,s)x$ is continuous for every $x\in X$.
\end{enumerate}
\end{definition}

\begin{definition}\label{d.rate}\rm
We say that an increasing  function $\mu : \mathbb{R}^{+}_{0} \to [1,+\infty)$ is a \emph{growth rate} if $\mu(0)=1$ and  $\lim\limits_{t\to +\infty} \mu(t) = +\infty$.
\end{definition}

Examples of growth rates are $e^t$, $t+1$, $t+\sqrt{t^2+1}$, $\dfrac{\log(e^t+1)}{\log{2}}$ etc.

\begin{lemma}\label{lem:eq}
Let $\mu : \mathbb{R}^{+}_{0} \to [1,+\infty)$ be a differentiable growth rate and $K>0$. The following statements are equivalent:
      \begin{enumerate}
         \item[\rm(i)] $\sup\limits_{t \geq 0} \dfrac{\mu'(t)}{\mu(t)} \leq K$;
         \item[\rm(ii)] $\mu(t) \leq \mu(t_0) e^{K(t-t_0)}$ for every $t \geq t_0 \geq 0$;
         \item[\rm(iii)] $\sup\limits_{t \geq 0} \dfrac{\mu(t+\delta)}{\mu(t)} \leq e^{K \delta}$ for every $\delta > 0$.
      \end{enumerate}
\end{lemma}
\begin{proof}
(i) $\Rightarrow$ (ii). Let $t\geq t_0\geq 0$ and $\tau\in[t_0,t]$. We have that
$\dfrac{\mu'(\tau)}{\mu(\tau)}\leq K,$
which is equivalent to
$$\dfrac{d}{d\tau} \log \mu(\tau)\leq K, \,\tau\in[t_0,t].$$
Integrating from $t_0$ to $t$ in the last inequality, we deduce that
$$\log \mu(t)-\log \mu(t_0)\leq K (t-t_0).$$
This implies that (ii) holds. Implication (ii) $\Rightarrow$ (iii) is straightforward.\\
(iii) $\Rightarrow$ (i). Let $t\geq 0$. By (iii) we have
\begin{align*}
\dfrac{\dfrac{\mu(t+\delta)-\mu(t)}{\delta}}{\mu(t)}&=\dfrac{\dfrac{\mu(t+\delta)}{\mu(t)}-1}{\delta}
\leq \dfrac{e^{K\delta}-1}{\delta},
\end{align*}
for all  $\delta>0$.
Letting $\delta\to 0$ in the relation above and using the fact that $\mu$ is a differentiable function, we obtain
$$\dfrac{\mu'(t)}{\mu(t)} \leq K, \text{ for all } t\geq 0,$$
and hence (i) holds.
\end{proof}

For example, for the growth rates considered above, relation (i) in Lemma \ref{lem:eq} is fulfilled.

\begin{definition}\label{d.proj}\rm
A strongly continuous function $P:\mathbb{R}^{+}_{0}\to\mathcal{B}(X)$ is called a \emph{projection valued function} if $$P^2(t)=P(t) \text{ for every } t\geq 0. $$
\end{definition}
Given a projection valued function $P:\mathbb{R}^{+}_{0}\to\mathcal{B}(X)$, we denote by $Q$ the complementary projection valued function, that is $Q(t)=\mathrm{Id}-P(t)$ for every $t\geq 0$.

\begin{definition}\label{d.comp}\rm
We say that a projection valued function $P:\mathbb{R}^{+}_{0}\to\mathcal{B}(X)$ is \emph{compatible} with an evolution operator $U:\Delta\to\mathcal{B}(X)$ if, for all $t\geq s\geq 0$, we have
\begin{enumerate}
\item $P(t)U(t,s)=U(t,s)P(s)$;

\item the restriction $U(t,s)|_{Q(s)X}:Q(s)X\to Q(t)X$ is an isomorphism and we denote its inverse by $U_{Q}(s,t)$.
\end{enumerate}
\end{definition}

\begin{remark}\label{UQ}
If $P:\mathbb{R}^{+}_{0}\to\mathcal{B}(X)$ is a projection valued function compatible with an evolution operator $U:\Delta\to\mathcal{B}(X)$, then for every $(t,s)\in \Delta$, it follows
\begin{enumerate}
\item[\rm(i)]  $U(t,s)U_{Q}(s,t)Q(t)=Q(t);$

\item[\rm(ii)] $U_{Q}(s,t)U(t,s)Q(s)=Q(s).$
\end{enumerate}
Moreover, for all $t\geq \tau\geq s\geq 0$, we have that
$U_Q(s,\tau)U_Q(\tau,t)=U_Q(s,t).$
\end{remark}

\begin{definition}\label{d.dich}\rm
Given a growth rate $\mu: \mathbb{R}^{+}_{0} \to [1,+\infty)$ and a projection valued function $P:\mathbb{R}^{+}_{0}\to\mathcal{B}(X)$ compatible with an evolution operator $U:\Delta\to\mathcal{B}(X)$, we say that $U$ has a \emph{nonuniform $\mu$-dichotomy} with projection valued function $P$ if there exist constants $a,b>0$,  $\varepsilon\geq 0$ and $N_1, N_2\geq 1$ such that, for all $t\geq s\geq 0$, we have
\begin{enumerate}
\item  $ \|U(t,s)P(s)\| \leq N_1 \left(\dfrac{\mu(t)}{\mu(s)}\right)^{-a} \mu(s)^{\varepsilon};$

\item  $ \|U_{Q}(s,t)Q(t)\| \leq N_2 \left(\dfrac{\mu(t)}{\mu(s)}\right)^{-b} \mu(t)^{\varepsilon}.$
\end{enumerate}

When $\varepsilon=0$, we say that $U$ has a \emph{uniform $\mu$-dichotomy} with projection valued function $P$.
\end{definition}

In the following we consider particular cases of the notion of nonuniform $\mu$-dichotomy:
\begin{enumerate}
\item if $\mu(t)=e^t$, then we recover the notion of nonuniform exponential dichotomy (in the sense of Barreira-Valls) \cite{Ba.Va.2008-1} and in particular (when $\varepsilon=0$) the classical notion of uniform exponential dichotomy;

\item if $\mu(t)=t+1$, then we recover the notion of nonuniform polynomial dichotomy  \cite{Ba.Va.2009,Be.Si.2009,Be.Si.2012}.
\end{enumerate}

More generally, we consider an evolution operator which has a nonuniform $\mu$-dichotomy with a projection valued function $P$, for an arbitrary continuous growth rate $\mu$:
\begin{example}\label{ex1}
Let $\mu: \mathbb{R}^{+}_{0} \to [1,+\infty)$
be a continuous growth rate and  $\varepsilon\geq 0$ be a non-negative real number. On  $X=\mathbb{R}^2$ endowed with the norm $\|(x_1,x_2)\|=\max\{|x_1|,|x_2|\},$ we consider the projection valued function
$$P(t)(x_1,x_2)=(x_1+(\mu(t)^{\varepsilon}-1)x_2,0)$$
and its complementary projection valued function
$$Q(t)(x_1,x_2)=((1-\mu(t)^{\varepsilon})x_2,x_2).$$
Obviously, we have that
\begin{equation}\label{eq:norm}
\|P(t)\|=\mu(t)^{\varepsilon} \text{ and } \|Q(t)\|= \max\{\mu(t)^{\varepsilon}-1,1\}\leq\mu(t)^{ \varepsilon}.
\end{equation}
Given $a,b>0$, we consider the evolution operator $U:\Delta\to\mathcal{B}(\mathbb{R}^2)$, $$U(t,s)=\left(\dfrac{\mu(t)}{\mu(s)}\right)^{-a}P(s)+\left(\dfrac{\mu(t)}{\mu(s)}\right)^{b} Q(t).$$
Since  $P(t)P(s)=P(s)$, $Q(t)Q(s)=Q(t)$ and $Q(t)P(s)=0$, we have that $P$ is a projection valued function compatible with $U$. Moreover, it follows that
$$U(t,s)P(s)=\left(\dfrac{\mu(t)}{\mu(s)}\right)^{-a}P(s)
\text{ and }
U_{Q}(s,t)Q(t)=\left(\dfrac{\mu(t)}{\mu(s)}\right)^{-b} Q(s).$$
By \eqref{eq:norm} and using the relations above, we deduce that
$$\|U(t,s)P(s)\|=\left(\dfrac{\mu(t)}{\mu(s)}\right)^{-a}\mu(s)^{\varepsilon}$$
 and
$$\|U_{Q}(s,t)Q(t)\|\leq\left(\dfrac{\mu(t)}{\mu(s)}\right)^{-b}\mu(s)^{\varepsilon}\leq \left(\dfrac{\mu(t)}{\mu(s)}\right)^{-b}\mu(t)^{\varepsilon},$$
for  $t\geq s\geq 0$, which shows that the evolution operator $U$ has a nonuniform $\mu$-dichotomy with projection valued function $P$.
Let now $\varepsilon>0$ and assume that $U$ has a uniform $\mu$-dichotomy with projection valued function $P$. Then there exist $\nu>0$ and $N\geq 1$ such that
\begin{equation*}\label{eq:unif.st}
\|U(t,s)P(s)\|\leq N \left(\dfrac{\mu(t)}{\mu(s)}\right)^{-\nu}, \text{ for all } t\geq s\geq 0,
\end{equation*}
which is equivalent to
\begin{equation}\label{eq:unif.st1}
\left(\dfrac{\mu(t)}{\mu(s)}\right)^{-a}\mu(s)^{ \varepsilon}\leq N \left(\dfrac{\mu(t)}{\mu(s)}\right)^{-\nu}, \text{ for all } t\geq s\geq 0.
\end{equation}
Setting $t=s$ in~\eqref{eq:unif.st1} we have
   $$ \mu(s)^{ \varepsilon} \leq N, \text{ for all } s\geq 0,$$
and this is absurd because $\lim\limits_{t \to +\infty} \mu(t) = + \infty$.
Therefore,  when $\varepsilon > 0$  the evolution operator $U$ does not have a uniform $\mu$-dichotomy with  projection valued function $P$.
\end{example}

Now, we provide an example of a linear differential equation that generates an evolution operator which admits a $\mu$-dichotomy.

\begin{example}
   Given a growth rate $\mu$ and $a\in \mathbb{R}$,  $\varepsilon\geq 0$, consider the functions
   \begin{equation*}
      \nu_{a,\eps}(t)
      = a \,\dfrac{\mu'(t)}{\mu(t)}
         + \dfrac{\eps}{2} \dfrac{\mu'(t)}{\mu(t)}\prts{\cos t - 1}
         - \dfrac{\eps}{2} \log\prts{\mu(t)} \sin t
   \end{equation*}
   and
   \begin{equation*}
      \sgm_{a,\eps}(t)
      = \mu(t)^a \e^{\eps\prtsr{\log\prts{\mu(t)} \prts{\cos t - 1}}/2}.
   \end{equation*}
   It is clear that for $t, s \ge 0$,
   \begin{equation} \label{ineq:smg_(a,eps)(t)/smg_(a,eps)(s)_le_...}
      \dfrac{\sgm_{a, \eps}(t)}{\sgm_{a,\eps}(s)}
      \le
         \pfrac{\mu(t)}{\mu(s)}^a \mu(s)^\eps.
   \end{equation}

   In $\ell_\infty=\{x=(x_1,x_2,x_3,\ldots,x_n,\ldots):\, \sup\limits_{n\in\mathbb{N}^*} |x_n|<+\infty\}$ consider the linear differential equation
   \begin{equation}\label{eq:v'(t)=A(t)v(t)}
      v'(t) = A(t) v(t)
   \end{equation}
   where $A(t) \colon \ell_\infty \to \ell_\infty$ is the operator given by
   \begin{align*}
      & A(t)\prts{x_1, x_2, x_3, \ldots , x_n, \ldots}\\
      & = \prts{\nu_{a_1,\eps_1}(t) \, x_1, \nu_{a_2,\eps_2}(t) \, x_2,
         \nu_{a_3,\eps_3}(t) \, x_3, \ldots, \nu_{a_n,\eps_n}(t) \, x_n, \ldots},
   \end{align*}
  with $\prts{a_n}_{n \in \mathbb{N}^*}$ and $\prts{\eps_n}_{n \in \mathbb{N}^*}$ be two bounded real sequences such that
   \begin{equation*}
      a_{2n} > 0, \ \ \
      a_{2n-1} < 0, \ \ \
      \eps_n \ge 0, \text{  for every $n \in \mathbb{N}^*$,}
   \end{equation*}
   and
   \begin{equation*}
      \sup_{n \in \mathbb{N}^*} a_{2n-1} < 0,
      \ \ \
      \inf_{n \in \mathbb{N}^*} a_{2n} > 0.
   \end{equation*}

   The evolution operator of equation~\eqref{eq:v'(t)=A(t)v(t)} is given by
   \begin{align*}
      & U(t,s) \prts{x_1, x_2, x_3, \ldots , x_n, \ldots}\\
      & = \prts{
         \dfrac{\sgm_{a_1,\eps_1}(t)}{\sgm_{a_1,\eps_1}(s)} \, x_1, \dfrac{\sgm_{a_2,\eps_2}(t)}{\sgm_{a_2,\eps_2}(s)} \, x_2, \dfrac{\sgm_{a_3,\eps_3}(t)}{\sgm_{a_3,\eps_3}(s)} \, x_3,
         \ldots,
         \dfrac{\sgm_{a_n,\eps_n}(t)}{\sgm_{a_n,\eps_n}(s)} \, x_n,
         \ldots}.
   \end{align*}
   With the projection $P(t) \colon \ell_\infty \to \ell_\infty$ defined by
   \begin{equation*}
      P(t)\prts{x_1, x_2, x_3, x_4, x_5, x_6, \ldots}
      = \prts{x_1, 0, x_3, 0, x_5, 0, \ldots},
   \end{equation*}
   it follows  from~\eqref{ineq:smg_(a,eps)(t)/smg_(a,eps)(s)_le_...} that
   \begin{align*}
      \|U(t,s)P(s)\|
      & = \sup_{n \in \mathbb{N}^*} \dfrac{\sgm_{a_{2n-1},\eps_{2n-1}}(t)}
         {\sgm_{a_{2n-1},\eps_{2n-1}}(s)}\\
      & \le \sup_{n \in \mathbb{N}^*} \pfrac{\mu(t)}{\mu(s)}^{a_{2n-1}}
         \mu(s)^{\eps_{2n-1}}\\
      & = \pfrac{\mu(t)}{\mu(s)}^{-a} \mu(s)^\eps,
   \end{align*}
   where $a = -\sup\limits_{n \in \mathbb{N}^*} a_{2n-1}$ and $\eps = \sup\limits_{n \in \mathbb{N}^*} \eps_n$, and
\begin{align*}
      \|U_Q(s,t)Q(t)\|
      & = \sup_{n \in \mathbb{N}^*} \dfrac{\sgm_{a_{2n},\eps_{2n}}(s)}
         {\sgm_{a_{2n},\eps_{2n}}(t)}\\
      & \le \sup_{n \in \mathbb{N}^*} \pfrac{\mu(s)}{\mu(t)}^{a_{2n}}
         \mu(t)^{\eps_{2n}}\\
      & = \pfrac{\mu(t)}{\mu(s)}^{-b} \mu(t)^\eps,
\end{align*}
with $\displaystyle b = \inf_{n \in \mathbb{N}^{*}} a_{2n}$. Clearly, the projection valued function $P$ is compatible with the evolution operator $U$ and thus the linear equation~\eqref{eq:v'(t)=A(t)v(t)} admits a $\mu$-dichotomy with projection valued function $P$.
\end{example}

\black

\section{The main results}

For a given evolution operator $U:\Delta\to\mathcal{B}(X)$ and a projection valued function $P:\mathbb{R}^{+}_{0}\to\mathcal{B}(X)$ compatible with $U$, we denote the \emph{Green function} associated to the evolution operator $U$ and the projection valued function $P$ compatible with $U$ by
\[
G(t,s):=\begin{cases}
\phantom{-}U(t,s)P(s), &\text{for } t> s\geq 0,\\
-U_{Q}(t,s)Q(s),       &\text{for } s> t\geq 0.
\end{cases}
\]

The following theorem gives a necessary condition for the existence of nonuniform $\mu$-dichotomy with a  projection valued function $P$.
 \begin{theorem} \label{thm:dicho=>ine_int}
      Let $p>0$ and $\mu: \mathbb{R}^{+}_{0} \to [1,+\infty)$ be a differentiable growth rate. If the evolution operator $U:\Delta\to\mathcal{B}(X)$ has a nonuniform $\mu$-dichotomy with a dichotomy projection valued function $P:\mathbb{R}^{+}_{0}\to\mathcal{B}(X)$, then for every constant $\gamma < \min\{a,b\}$ it follows that
      \begin{equation}\label{eq:1}
         \int_0^{+\infty} \dfrac{\mu'(\tau)}{\mu(\tau)}
            \left(\dfrac{\mu(\tau)}{\mu(t)}\right)^{p \gamma  \, \mathrm{sign}(\tau-t)} \|G(\tau,t)x\|^p d\tau
         \leq D \mu(t)^{p \varepsilon} \| x\|^p,
      \end{equation}
      for  every $(t,x)\in \mathbb{R}^{+}_{0}\times X$, where $a,b>0$,  $\varepsilon\geq 0$ and $N_1, N_2\geq 1$ are given by Definition \ref{d.dich}, and
 \begin{equation}\label{eq:D}
 D=\dfrac{N_1^p}{p(a-\gamma)}+\cfrac{N_2^p}{p(b-\gamma)} \,\cdot
 \end{equation}

 \end{theorem}

    \begin{proof}
     For $(t,x)\in \mathbb{R}^{+}_{0}\times X$ and  $\gamma < \min\{a,b\}$, we have
      \begin{align*}
       \int_0^{+\infty} &\dfrac{\mu'(\tau)}{\mu(\tau)}
            \left(\dfrac{\mu(\tau)}{\mu(t)}\right)^{p \gamma \,\mathrm{sign}(\tau-t)} \|G(\tau,t)x\|^p d\tau\\
          & =   \int_t^{+\infty} \dfrac{\mu'(\tau)}{\mu(\tau)}
            \left(\dfrac{\mu(\tau)}{\mu(t)}\right)^{p \gamma} \|U(\tau,t)P(t)x\|^p d\tau\\
          &\quad +  \int_0^{t} \dfrac{\mu'(\tau)}{\mu(\tau)}
            \left(\dfrac{\mu(t)}{\mu(\tau)}\right)^{p \gamma} \|U_{Q}(\tau,t)Q(t)x\|^p d\tau \\
          & \leq  N_1^p \mu(t)^{p \varepsilon}\mu(t)^{p(a-\gamma)}\int_t^{+\infty}  \mu(\tau)^{-p(a-\gamma)-1}\mu'(\tau) d\tau \, \|x\|^p\\
          &\quad +  N_2^p \mu(t)^{p \varepsilon}\mu(t)^{-p(b-\gamma)}\int_0^{t} \mu(\tau)^{p(b-\gamma)-1}\mu'(\tau)  d\tau \, \|x\|^p\\
          & \leq D \mu(t)^{p \varepsilon} \|x \|^p,
      \end{align*}
      and this proves the result.
   \end{proof}

The next result is a partial converse of Theorem \ref{thm:dicho=>ine_int}. It can be considered a Datko type theorem for the existence of nonuniform $\mu$-dichotomy.
\begin{theorem}\label{Th.Datko.s}
      Let $\mu : \mathbb{R}^{+}_{0} \to [1,+\infty)$  be a differentiable growth rate such that
      \begin{equation}\label{eq:bound}
         K_\mu := \sup\limits_{t \geq 0} \dfrac{\mu'(t)}{\mu(t)} < + \infty.
      \end{equation}
      Assume that $U:\Delta\to\mathcal{B}(X)$ is an evolution operator and $P:\mathbb{R}^{+}_{0}\to\mathcal{B}(X)$ is a projection valued function compatible with $U$ such that
      \begin{equation}\label{eq:growth}
      \| G(t,s) \|\leq M \dfrac{\mu'(s)}{\mu(s)} \left(\dfrac{\mu(t)}{\mu(s)}\right)^{\omega \, \mathrm{sign} (t-s)} \mu(s)^\alpha, \,
       t,s\geq 0, \, t\neq s,
      \end{equation}
      for some  $\omega>0$, $\alpha\geq 0$ and $M\geq 1$.
      If~\eqref{eq:1} holds for some $p,D \geq 1$, $\gamma > \alpha$ and $\varepsilon\geq 0$, then $U$ has a nonuniform $\mu$-dichotomy with projection valued function $P$.
   \end{theorem}

   \begin{proof}
     Let $x \in X$.  By ~\eqref{eq:bound},~\eqref{eq:growth} and Lemma \ref{lem:eq}, if $t \geq s+1$, we have
       \begin{equation}\label{eq:maj-1}
         \begin{split}
            & \|U(t,s)P(s)x\|^p
               = \int_{t-1}^t \|U(t,s)P(s)x\|^p d\tau \\
            & \leq M^p \int_{t-1}^t  \left(\dfrac{\mu'(\tau)}{\mu(\tau)}\right)^p
               \left(\dfrac{\mu(t)}{\mu(\tau)}\right)^{p \omega}
               \mu(\tau)^{p \alpha} \|U(\tau,s)P(s)x\|^p d\tau \\
            & \leq M^p K_\mu^{p-1}\mu(t)^{p \alpha}\\
              &\quad \times\int_{t-1}^t \dfrac{\mu'(\tau)}{\mu(\tau)} \left(\dfrac{\mu(t)}{\mu(\tau)}\right)^{p \omega}
                \left(\dfrac{\mu(\tau)}{\mu(s)}\right)^{-p \gamma} \left(\dfrac{\mu(\tau)}{\mu(s)}\right)^{p \gamma} \|U(\tau,s)P(s)x\|^p d\tau  \\
            & = M^p K_\mu^{p-1}\mu(t)^{p \alpha} \left(\dfrac{\mu(t)}{\mu(s)}\right)^{-p \gamma}\\
              &\quad \times\int_{t-1}^t \dfrac{\mu'(\tau)}{\mu(\tau)} \left(\dfrac{\mu(t)}{\mu(\tau)}\right)^{p(\omega+\gamma)}
               \left(\dfrac{\mu(\tau)}{\mu(s)}\right)^{p \gamma} \|U(\tau,s)P(s)x\|^p d\tau  \\
            & \leq M^p K_\mu^{p-1}e^{K_\mu(\omega+\gamma)p} \mu(t)^{p \alpha} \left(\dfrac{\mu(t)}{\mu(s)}\right)^{-p \gamma}\\
              &\quad \times\int_{t-1}^t \dfrac{\mu'(\tau)}{\mu(\tau)}
               \left(\dfrac{\mu(\tau)}{\mu(s)}\right)^{p \gamma} \|U(\tau,s)P(s)x\|^p d\tau  \\
            & \leq M^p K_\mu^{p-1} e^{K_\mu(\omega+\gamma)p} \mu(t)^{p \alpha} \left(\dfrac{\mu(t)}{\mu(s)}\right)^{-p \gamma}\\
              &\quad \times\int_{s}^{\infty} \dfrac{\mu'(\tau)}{\mu(\tau)}
               \left(\dfrac{\mu(\tau)}{\mu(s)}\right)^{p \gamma} \|U(\tau,s)P(s)x\|^p d\tau  \\
            & \leq D M^p K_\mu^{p-1} e^{K_\mu(\omega+\gamma)p} \mu(t)^{p \alpha}
               \left(\dfrac{\mu(t)}{\mu(s)}\right)^{-p \gamma} \mu(s)^{p \varepsilon} \| x\|^{p}\\
            & = D M^p K_\mu^{p-1} e^{K_\mu(\omega+\gamma)p}
               \left(\dfrac{\mu(t)}{\mu(s)}\right)^{-p(\gamma-\alpha)} \mu(s)^{p(\varepsilon+\alpha)} \| x\|^{p}
         \end{split}
      \end{equation}
      and, if $s \leq t < s+1$, we get
      \begin{equation}\label{eq:maj-2}
         \begin{split}
            \|U(t,s)P(s)x\|
            & \leq M \dfrac{\mu'(s)}{\mu(s)} \left(\dfrac{\mu(t)}{\mu(s)}\right)^\omega
               \mu(s)^\alpha \|x\|\\
            & \leq M K_\mu e^{K_\mu(\omega + \gamma -\alpha)}
               \left(\dfrac{\mu(t)}{\mu(s)}\right)^{-(\gamma-\alpha)}
               \mu(s)^{\varepsilon+\alpha} \|x\|.
         \end{split}
      \end{equation}
On the other hand, for $t \geq s+1$ we have
       \begin{equation}\label{eq:maj-3}
         \begin{split}
            & \|U_{Q}(s,t)Q(t)x\|^p
               = \int_{s}^{s+1} \|U_{Q}(s,t)Q(t)x\|^p d\tau \\
            & \leq M^p \int_{s}^{s+1}  \left(\dfrac{\mu'(\tau)}{\mu(\tau)}\right)^p
               \left(\dfrac{\mu(\tau)}{\mu(s)}\right)^{p \omega}
               \mu(\tau)^{p \alpha} \|U_{Q}(\tau,t)Q(t)x\|^p d\tau \\
            & \leq M^p K_\mu^{p-1}\mu(s+1)^{p \alpha}\\
              &\quad \times\int_{s}^{s+1} \dfrac{\mu'(\tau)}{\mu(\tau)} \left(\dfrac{\mu(\tau)}{\mu(s)}\right)^{p \omega}
                \left(\dfrac{\mu(t)}{\mu(\tau)}\right)^{-p\gamma} \left(\dfrac{\mu(t)}{\mu(\tau)}\right)^{p \gamma} \|U_{Q}(\tau,t)Q(t)x\|^p d\tau  \\
            & \leq M^p K_\mu^{p-1}e^{K_{\mu}\alpha p}\mu(s)^{p \alpha} \left(\dfrac{\mu(t)}{\mu(s)}\right)^{-p \gamma}\\
              &\quad \times\int_{s}^{s+1} \dfrac{\mu'(\tau)}{\mu(\tau)} \left(\dfrac{\mu(\tau)}{\mu(s)}\right)^{p(\omega+\gamma)}
               \left(\dfrac{\mu(t)}{\mu(\tau)}\right)^{p \gamma} \|U_{Q}(\tau,t)Q(t)x\|^p d\tau  \\
            & \leq M^p K_\mu^{p-1} e^{K_\mu(\alpha+\omega+\gamma)p} \mu(s)^{p \alpha} \left(\dfrac{\mu(t)}{\mu(s)}\right)^{-p \gamma}\\
              &\quad \times\int_{0}^{t} \dfrac{\mu'(\tau)}{\mu(\tau)}
               \left(\dfrac{\mu(t)}{\mu(\tau)}\right)^{p \gamma} \|U_{Q}(\tau,t)Q(t)x\|^p d\tau  \\
            & \leq D M^p K_\mu^{p-1} e^{K_\mu(\alpha+\omega+\gamma)p}
               \left(\dfrac{\mu(t)}{\mu(s)}\right)^{-p(\gamma+\alpha)} \mu(t)^{p(\varepsilon+\alpha)} \| x\|^{p}
         \end{split}
      \end{equation}
      and, for $s \leq t < s+1$ we get
      \begin{equation}\label{eq:maj-4}
         \begin{split}
            \|U_{Q}(s,t)Q(t)x\|
            & \leq M K_\mu e^{K_\mu(\omega + \gamma +\alpha)}
               \left(\dfrac{\mu(t)}{\mu(s)}\right)^{-(\gamma+\alpha)}
               \mu(t)^{\varepsilon+\alpha} \|x\|.
         \end{split}
      \end{equation}

     From~\eqref{eq:maj-1}, \eqref{eq:maj-2}, \eqref{eq:maj-3} and \eqref{eq:maj-4} we conclude that the evolution operator $U$ has a nonuniform $\mu$-dichotomy with projection valued function $P$.
    \end{proof}

 In the particular case when $\mu(t)=e^t$, we recover Theorem 1 in \cite{Lu.Me.2014}:
  \begin{corollary}\label{cor.exponential}
 Assume  that  $U:\Delta\to\mathcal{B}(X)$ is an evolution operator and $P:\mathbb{R}^{+}_{0}\to\mathcal{B}(X)$ is a projection valued function compatible with $U$ such that there exist constants  $\omega>0$, $\alpha\geq 0$ and $M\geq 1$ with
      \begin{equation*}\label{eq:growth.exp}
      \| G(t,s) \|\leq M  e^{\alpha s} e^{\omega|t-s|}, \text{ for } t,s\geq 0, \, t\neq s.
      \end{equation*}
  If there exist $p,D \geq 1$,  $\gamma > \alpha$ and $\varepsilon\geq 0$ such that
       \begin{equation*}
         \int_0^{+\infty} e^{p \gamma| \tau-t|} \|G(\tau,t)x\|^p d\tau
         \leq D e^{p\varepsilon t} \| x\|^p,
      \end{equation*}
      for every $(t,x)\in \mathbb{R}^{+}_{0} \times X$, then $U$ has a nonuniform exponential dichotomy with  projection valued function $P$.
  \end{corollary}

A similar result to the one above can be obtained in the case of nonuniform polynomial dichotomy:
  \begin{corollary}\label{cor.polynomial}
  We assume that $U:\Delta\to\mathcal{B}(X)$ is an evolution operator and $P:\mathbb{R}^{+}_{0}\to\mathcal{B}(X)$ is a projection valued function compatible with $U$ such that there exist constants $\omega>0$, $\alpha\geq 0$ and $M\geq 1$ with
  \begin{equation*}\label{eq:growth.pol}
      \| G(t,s) \|\leq M \dfrac{1}{s+1}
      \left(\dfrac{t+1}{s+1}\right)^{\omega \, \mathrm{sign} (t-s)} (s+1)^\alpha, \text{ for } t,s\geq 0, \, t\neq s.
  \end{equation*}
  If there exist $p,D \geq 1$,  $\gamma > \alpha$ and $\varepsilon\geq 0$  such that
       \begin{equation*}
         \int_0^{+\infty} \dfrac{1}{\tau+1} \left(\dfrac{\tau+1}{t+1}\right)^{p \gamma  \, \mathrm{sign}(\tau-t)} \|G(\tau,t)x\|^p d\tau
         \leq D (t+1)^{p \varepsilon} \| x\|^p,
      \end{equation*}
      for every $(t,x)\in \mathbb{ R}^{+}_{0} \times X$, then $U$ has a nonuniform polynomial dichotomy with projection valued function $P$.
  \end{corollary}

In the following example we consider an evolution operator that has a nonuniform $\mu$-dichotomy for a given growth rate $\mu$, different from both exponential and polynomial functions.
\begin{example}\label{ex2}
Let $\mu(t)=t+\sqrt{t^2+1}$, $t\geq 0$. Given $a,b>1$ and $\alpha\geq 0$ with $\alpha+1<\min\{a,b\}$,  consider the evolution operator $U:\Delta\to\mathcal{B}(\mathbb{R}^2)$, $U(t,s)(x_1,x_2)=\left(U_1(t,s)x_1,U_2(t,s)x_2\right),$ where
\begin{align*}
U_1(t,s)x_1&=\frac{\mu'(s)}{\mu'(t)}\left(\frac{\mu(t)}{\mu(s)}\right)^{- a} e^{\alpha \sin^2 s\log \mu(s)-\alpha \sin^2 t\log \mu(t)}x_1,\\
U_2(t,s)x_2&=\frac{\mu'(s)}{\mu'(t)}\left(\frac{\mu(s)}{\mu(t)}\right)^{- b} e^{\alpha \sin^2 t\log \mu(t)-\alpha \sin^2 s\log \mu(s)}x_2.
\end{align*}
Obviously $\mu$ is a differentiable growth rate with
\begin{equation}\label{eq:ex}
\mu'(t)=1+\frac{t}{\sqrt{t^2+1}}\geq 1 \;\text{ and }\; \frac{\mu'(t)}{\mu(t)}=\frac{1}{\sqrt{t^2+1}}\leq 1, \text{ for all } t\geq 0.
\end{equation}
Using Theorem \ref{Th.Datko.s}, we will prove that $U$ has a nonuniform $\mu$-dichotomy with the projection valued function $P(t)(x_1,x_2)=(x_1,0)$. Indeed, by \eqref{eq:ex} we have that
$$\|G(t,s)\|\leq \frac{\mu'(s)}{\mu(s)}\left(\frac{\mu(t)}{\mu(s)}\right)^{\omega\,\mathrm{sign}(t-s)}\mu(s)^{\alpha+1}, \, t,s\geq 0, \, t\neq s,$$
for each $\omega>0$.

Furthermore, proceeding in a similar manner to the proof of Theorem  \ref{thm:dicho=>ine_int}, we obtain
$$\int_0^{+\infty} \dfrac{\mu'(\tau)}{\mu(\tau)}
            \left(\dfrac{\mu(\tau)}{\mu(t)}\right)^{\gamma  \, \mathrm{sign}(\tau-t)} \|G(\tau,t)x\| d\tau
         \leq D\mu(t)^{\alpha+1} \| x\|, \;(t,x)\in \mathbb{R}^{+}_{0}\times \mathbb{R}^2,$$
for $\gamma\in (\alpha+1,\min\{a,b\})$ and $D=\max\left\{1,\dfrac{1}{a-\gamma}+\dfrac{1}{b-\gamma}\right\}$, which shows that $U$ has a nonuniform $\mu$-dichotomy
with  projection valued function $P$.
\end{example}

\medskip

We now discuss the relation between Lyapunov functions and nonuniform $\mu$-dichotomies.

Given a differentiable growth rate $\mu : \mathbb{R}^{+}_{0} \to [1,+\infty)$, a projection valued function $P:\mathbb{R}^{+}_{0}\to\mathcal{B}(X)$ and constants $\gamma > 0$ and $p \ge 1$, we denote by $\mathcal{H}_{\gamma,p}^{\mu}(P)$ the set of all strongly continuous operator-valued functions
      $$ H : \mathbb{R}^{+}_{0} \to \mathcal{B}(X)$$
   such that
   \begin{equation}\label{eq:Ly}
       \|H(t)x\|
         \leq \left(\dfrac{\mu'(t)}{\mu(t)}\right)^{1/p} \mu(t)^\gamma \|P(t) x\|
            + \left(\dfrac{\mu'(t)}{\mu(t)}\right)^{1/p} \mu(t)^{-\gamma} \|Q(t) x\|,
   \end{equation}
   for every $ t \geq 0$ and every $x \in X$.

Next, we define Lyapunov functions adapted to our situation.
Given an evolution operator $U:\Delta\to\mathcal{B}(X)$, a projection valued function $P:\mathbb{R}^{+}_{0}\to\mathcal{B}(X)$ compatible with $U$, a growth rate $\mu: \mathbb{R}^{+}_{0} \to [1,+\infty)$ and a constant $\gamma>0$, we say that a function $L : \mathbb{R}^{+}_{0} \times X \to \mathbb{R}$ is a \emph{Lyapunov function for $U$ with projection valued function $P$, growth rate $\mu$ and exponent $\gamma$} if there are constants $p, D \geq 1$ and $\varepsilon \geq 0$  such that for all strongly continuous operator valued functions $H \in \mathcal H_{\gamma,p}^\mu(P)$ and all $(t,s,x)\in\Delta\times X$, we have
      \begin{enumerate}
         \item[\rm(i)] $L(t, U(t,s)x) + \displaystyle\int_s^t \|H(\tau) U(\tau,s) x\|^p d\tau \leq L(s,x);$
         \item[\rm(ii)] $L(t,P(t)x) \geq 0$ and $L(t,Q(t) x) \leq 0$;
         \item[\rm(iii)] $\mu(t)^{-p\gamma} L(t,P(t) x)- \mu(t)^{p\gamma} L(t,Q(t) x)\leq 2^{p-1} D \mu(t)^{p\varepsilon} \|x\|^p$.
      \end{enumerate}

\begin{theorem}\label{th:Ly.n}
      Let $\mu : \mathbb{R}^{+}_{0} \to [1,+\infty)$ be a differentiable growth rate. If the evolution operator $U:\Delta\to\mathcal{B}(X)$ has a nonuniform $\mu$-dichotomy with a projection valued function $P:\mathbb{R}^{+}_{0}\to\mathcal{B}(X)$, then for every positive constant $\gamma < \min\{a,b\}$, there is a Lyapunov function for $U$ with projection valued function $P$, growth rate $\mu$ and exponent $\gamma$.
\end{theorem}

 \begin{proof}
 Let $p\geq 1$ and $H \in \mathcal H_{\gamma,p}^\mu(P)$. We define
 $$ L(t,x) := 2^{p-1}\left( \int_t^{+\infty} \|H(\tau) U(\tau,t) P(t) x\|^p d\tau
               -  \int_0^t \|H(r) U_{Q}(\tau,t)Q(t)x\|^p d\tau\right).$$
 We have
 \begin{equation}\label{eq:A}
 \begin{split}
  L(t, U(t,s)x) & + \int_s^t \|H(\tau) U(\tau,s) x\|^p d\tau\\
 &=2^{p-1} \int_t^{+\infty} \|H(\tau) U(\tau,s) P(s) x\|^p d\tau\\
 & \quad - 2^{p-1} \int_0^t \|H(\tau) U_{Q}(\tau,t)Q(t)U(t,s)x\|^p d\tau\\
 & \quad  +\int_s^t \|H(\tau) U(\tau,s) x\|^p d\tau.
 \end{split}
 \end{equation}
 We first compute
 \begin{align*}
 &\int_0^t \|H(\tau) U_{Q}(\tau,t)Q(t)U(t,s)x\|^p d\tau\\
 &= \int_0^s \|H(\tau) U_{Q}(\tau,t)Q(t)U(t,s)x\|^p d\tau + \int_s^t \|H(\tau) U_{Q}(\tau,t)Q(t)U(t,s)x\|^p d\tau\\
 &= \int_0^s \|H(\tau) U_{Q}(\tau,s)Q(s)x\|^p d\tau + \int_s^t \|H(\tau) U(\tau,s)Q(s)x\|^p d\tau.
 \end{align*}
On the other hand, using the inequality
$$\| x+y\|^p\leq 2^{p-1}\|x\|^p+ 2^{p-1}\|y\|^p, \text{ for } x,y\in X,$$
it follows that
 \begin{equation}\label{eq:B}
 \begin{split}
&\int_s^t \|H(\tau) U(\tau,s) x\|^p d\tau\\
&=\int_s^t \|H(\tau) U(\tau,s)P(s) x + H(\tau) U(\tau,s)Q(s) x\|^p d\tau\\
&\leq 2^{p-1} \int_s^{t} \|H(\tau) U(\tau,s) P(s) x\|^p d\tau+2^{p-1} \int_s^{t} \|H(\tau) U(\tau,s) Q(s) x\|^p d\tau.
 \end{split}
 \end{equation}
Now,   by~\eqref{eq:A} and~\eqref{eq:B} we have
\begin{align*}
&L(t, U(t,s)x)  + \int_s^t \|H(\tau) U(\tau,s) x\|^p d\tau\\
&\leq 2^{p-1} \int_s^{+\infty} \|H(\tau) U(\tau,s) P(s) x\|^p d\tau
               - 2^{p-1} \int_0^s \|H(\tau) U_{Q}(\tau,s)Q(s)x\|^p d\tau \\
&= L(s,x),
\end{align*}
for all $(t,s,x)\in\Delta\times X$.
Clearly
         $ L(t,P(t) x) \geq 0$      and
         $ L(t,Q(t) x) \leq 0.$

Moreover, by Theorem~\ref{thm:dicho=>ine_int} we deduce that

     \begin{align*}
         \mu(t)^{-p\gamma} & L(t,P(t) x)- \mu(t)^{p\gamma} L(t,Q(t) x)\\
         & = 2^{p-1} \mu(t)^{-p\gamma}\int_t^{+\infty} \|H(\tau) U(\tau,t) P(t) x\|^p  d\tau\\
           &\quad + 2^{p-1} \mu(t)^{p\gamma} \int_0^t \|H(\tau) U_{Q}(\tau,t)Q(t)x\|^p d\tau\\
         & \leq 2^{p-1}  \int_0^{+\infty}
            \dfrac{\mu'(\tau)}{\mu(\tau)} \left(\dfrac{\mu(\tau)}{\mu(t)}\right)^{p \gamma \,\mathrm{sign}(\tau-t)} \|G(\tau,t)  x\|^p d\tau\\
         & \leq 2^{p-1} D \mu(t)^{p\varepsilon} \| x\|^p,
      \end{align*}
for all  $(t,x)\in \mathbb{R}^{+}_{0}\times X$, where $D=\max\left\{1,\dfrac{N_1^p}{p(a-\gamma)}+\cfrac{N_2^p}{p(b-\gamma)}\right\}$ and $\varepsilon\geq 0$ is given  by Definition \ref{d.dich}.  Therefore, $L$ is a Lyapunov function for $U$ with projection valued function $P$, growth rate $\mu$ and exponent $\gamma$.
\end{proof}

The next theorem establishes the existence of nonuniform $\mu$-dichotomies assuming the existence of suitable Lyapunov functions and thus is a converse of Theorem~\ref{th:Ly.n}.

\begin{theorem}\label{th:Ly.s}
      Let $\mu : \mathbb{R}^{+}_{0} \to [1,+\infty)$  be a differentiable growth rate that satisfies \eqref{eq:bound}.  Assume  that $U:\Delta\to\mathcal{B}(X)$ is an evolution operator and $P:\mathbb{R}^{+}_{0}\to\mathcal{B}(X)$ is a projection valued function compatible with $U$ such that \eqref{eq:growth} holds for some $\alpha \geq 0$. If there is a Lyapunov function for $U$ with projection valued function $P$, growth rate $\mu$ and exponent $\gamma > \alpha$, then $U$ has a nonuniform $\mu$-dichotomy with projection valued function $P$.
\end{theorem}

   \begin{proof} Let $L : \mathbb{R}^{+}_{0} \times X \to \mathbb{R}$ be a Lyapunov function for $U$ with projection valued function $P$, growth rate $\mu$ and exponent $\gamma > \alpha$. Then there are constants $p, D \geq 1$ and $\varepsilon \geq 0$ such that relations (i)--(iii) from the definition of the Lyapunov functions hold
   for each $H \in \mathcal H_{\gamma,p}^\mu(P)$.
   Consider
         $$ H(t) x
         = \left(\dfrac{\mu'(t)}{\mu(t)}\right)^{1/p} \mu(t)^\gamma P(t) x
            + \left(\dfrac{\mu'(t)}{\mu(t)}\right)^{1/p} \mu(t)^{-\gamma} Q(t) x.$$
      It is easy to see that $H \in \mathcal{H}_{\gamma,p}^{\mu}(P)$. Thus, by (i) and (ii) mentioned above, we have
      \begin{align*}
         \int_t^u
            \dfrac{\mu'(\tau)}{\mu(\tau)} &\left(\dfrac{\mu(\tau)}{\mu(t)}\right)^{p\gamma} \|U(\tau,t) P(t) x\|^p  d\tau\\
         & = \mu(t)^{-p \gamma} \int_t^u \|H(\tau) U(\tau,t) P(t) x\|^p  d\tau\\
         & \leq \mu(t)^{-p \gamma} \left[L(t,P(t)x) - L(u,U(u,t) P(t) x)\right]\\
         & \leq \mu(t)^{-p \gamma} L(t,P(t)x),
      \end{align*}
   for every  $u\geq t$, which implies
   \begin{equation}\label{eq:Ly1}
         \int_t^{+\infty}
            \dfrac{\mu'(\tau)}{\mu(\tau)} \left(\dfrac{\mu(\tau)}{\mu(t)}\right)^{p\gamma} \|U(\tau,t) P(t) x\|^p  d\tau
            \leq \mu(t)^{-p \gamma} L(t,P(t)x).
   \end{equation}
   On the other hand, we get
   \begin{equation}\label{eq:Ly2}
      \begin{split}
         \int_0^t
            \dfrac{\mu'(\tau)}{\mu(\tau)} &\left(\dfrac{\mu(t)}{\mu(\tau)}\right)^{p\gamma} \|U_{Q}(\tau,t) Q(t) x\|^p  d\tau\\
            & = \mu(t)^{p \gamma} \int_0^t \|H(\tau) U_{Q}(\tau,t) Q(t) x\|^p  d\tau\\
         & \leq \mu(t)^{p \gamma} \left[L(0,U_{Q}(0,t)Q(t)x) - L(t, Q(t) x)\right]\\
         & \leq \mu(t)^{p \gamma} |L(t,Q(t)x)|.
      \end{split}
   \end{equation}
   By \eqref{eq:Ly1},  \eqref{eq:Ly2} and using (iii), it follows
   \begin{align*}
   \int_0^{+\infty}
            \dfrac{\mu'(\tau)}{\mu(\tau)} & \left(\dfrac{\mu(\tau)}{\mu(t)}\right)^{p\gamma\,\mathrm{sign}(\tau-t)} \|G(\tau,t)  x\|^p  d\tau\\
   & \leq \mu(t)^{-p \gamma} L(t,P(t)x)-\mu(t)^{p \gamma} L(t,Q(t)x)\\
   & \leq 2^{p-1} D \mu(t)^{p\varepsilon} \|x\|^p.
   \end{align*}
   Thus, by Theorem \ref{Th.Datko.s} we deduce that $U$ has a nonuniform $\mu$-dichotomy with projection valued function $P$.
   \end{proof}

If $X$ is a Hilbert space, we obtain the following result:

\begin{corollary}
Let $\mu : \mathbb{R}^{+}_{0} \to [1,+\infty)$  be a differentiable growth rate that satisfies \eqref{eq:bound}.  Assume that $U:\Delta\to\mathcal{B}(X)$ is an evolution operator on a Hilbert space $X$ and $P:\mathbb{R}^{+}_{0}\to\mathcal{B}(X)$ is a projection valued function compatible with $U$ such that \eqref{eq:growth} holds for some $\alpha\geq 0$.
If for some $\gamma > \alpha$  there is a strongly continuous operator valued function $W : \mathbb{R}^{+}_{0}  \to \mathcal{B}(X)$ with $W(t)^*=W(t)$ such that, for all strongly continuous operator valued functions $H \in \mathcal H_{\gamma,2}^\mu(P)$ and all $(t,s,x)\in\Delta\times X$, we have
\begin{equation*}
\langle U(t,s)^{*}W(t)U(t,s)x+\int_{s}^{t}U(\tau,s)^{*}H(\tau)^{*}H(\tau)U(\tau,s)x\;d\tau,x\rangle
\leq \langle W(s)x,x\rangle,
\end{equation*}
$$\langle W(t)P(t)x,P(t)x\rangle\geq 0,\quad  \langle W(t)Q(t)x,Q(t)x\rangle\leq 0,$$
and
\begin{align*}
\mu(t)^{- 2\gamma}\langle W(t)P(t)x,P(t)x\rangle  -\mu(t)^{2\gamma}\langle W(t)Q(t)x,&Q(t)x\rangle
&\leq  D \mu(t)^{2\varepsilon} \|x\|^2,
\end{align*}
for some  $D\geq 1$ and $\varepsilon\geq 0$, then $U$ has a nonuniform $\mu$-dichotomy with projection valued function $P$.
\end{corollary}
\begin{proof}
Set
$$L(t,x)=\langle W(t)x,x\rangle,\text{ for }(t,x)\in \mathbb{R}^{+}_{0}\times X. $$
It is easy to see that $L$ is a Lyapunov function for $U$ with projection valued function $P$, growth rate $\mu$ and exponent $\gamma$. Thus, by Theorem \ref{th:Ly.s} it follows that $U$ has a nonuniform $\mu$-dichotomy with projection valued function $P$.
\end{proof}

\section*{Acknowledgments}
The work of A. Bento and C. Silva was partially supported by FCT though Centro de Ma\-te\-m\'a\-ti\-ca e Aplica\c{c}\~oes da Universidade da Beira Interior (project PEst-OE/MAT/UI0212/2013).
The work of N. Lupa has been supported  by a grant of the Romanian National Authority for Scientific Research, CNCS-UEFISCDI, project number PN-II-ID-JRP-2011-2.

\def\cprime{$'$}


\begin{thebibliography}{10}
\expandafter\ifx\csname url\endcsname\relax
  \def\url#1{\texttt{#1}}\fi
\expandafter\ifx\csname urlprefix\endcsname\relax\def\urlprefix{URL }\fi

\bibitem{Ba.Me.Po}
M.~G. Babu{\c{t}}ia, M.~Megan, I.-L. Popa, On {$(h,k)$}-dichotomies for
  nonautonomous linear difference equations in {B}anach spaces, Int. J. Differ.
  Equ. (2013) Art. ID 761680, 7 pages.
\newline\urlprefix\url{http://dx.doi.org/10.1155/2013/761680}

\bibitem{Ba.Ch.Va.2013}
L.~Barreira, J.~Chu, C.~Valls, Lyapunov functions for general nonuniform
  dichotomies, Milan J. Math. 81~(1) (2013) 153--169.
\newline\urlprefix\url{http://dx.doi.org/10.1007/s00032-013-0198-y}

\bibitem{Ba.Va.2008-2}
L.~Barreira, C.~Valls, Growth rates and nonuniform hyperbolicity, Discrete
  Contin. Dyn. Syst. 22~(3) (2008) 509--528.
\newline\urlprefix\url{http://dx.doi.org/10.3934/dcds.2008.22.509}

\bibitem{Ba.Va.2008-1}
L.~Barreira, C.~Valls, Stability of nonautonomous differential equations, vol.
  1926 of Lecture Notes in Mathematics, Springer, Berlin, 2008.
\newline\urlprefix\url{http://dx.doi.org/10.1007/978-3-540-74775-8}

\bibitem{Ba.Va.2009}
L.~Barreira, C.~Valls, Polynomial growth rates, Nonlinear Anal. 71~(11) (2009)
  5208--5219.
\newline\urlprefix\url{http://dx.doi.org/10.1016/j.na.2009.04.005}

\bibitem{Ba.Va.2009-1}
L.~Barreira, C.~Valls, Quadratic {L}yapunov functions and nonuniform
  exponential dichotomies, J. Differential Equations 246~(3) (2009) 1235--1263.
\newline\urlprefix\url{http://dx.doi.org/10.1016/j.jde.2008.06.008}

\bibitem{Be.Si.}
A.~J.~G. {Bento}, C.~{Silva}, {Nonautonomous equations, generalized dichotomies
  and stable manifolds}, ArXiv e-prints.
\newline\urlprefix\url{http://arxiv.org/abs/0905.4935}

\bibitem{Be.Si.2009}
A.~J.~G. Bento, C.~Silva, Stable manifolds for nonuniform polynomial
  dichotomies, J. Funct. Anal. 257~(1) (2009) 122--148.
\newline\urlprefix\url{http://dx.doi.org/10.1016/j.jfa.2009.01.032}

\bibitem{Be.Si.2012}
A.~J.~G. Bento, C.~M. Silva, Stable manifolds for non-autonomous equations with
  non-uniform polynomial dichotomies, Q. J. Math. 63~(2) (2012) 275--308.
\newline\urlprefix\url{http://dx.doi.org/10.1093/qmath/haq047}

\bibitem{Be.Si.2013}
A.~J.~G. Bento, C.~M. Silva, Generalized nonuniform dichotomies and local
  stable manifolds, J. Dynam. Differential Equations 25~(4) (2013) 1139--1158.
\newline\urlprefix\url{http://dx.doi.org/10.1007/s10884-013-9331-4}

\bibitem{Bu.Ha.1989}
T.~Burton, L.~Hatvani, Stability theorems for nonautonomous
  functional-differential equations by {L}iapunov functionals, Tohoku Math. J.
  (2) 41~(1) (1989) 65--104.
\newline\urlprefix\url{http://dx.doi.org/10.2748/tmj/1178227868}

\bibitem{Bu.Ha.1990}
T.~A. Burton, L.~Hatvani, On nonuniform asymptotic stability for nonautonomous
  functional-differential equations, Differential Integral Equations 3~(2)
  (1990) 285--293.
\newline\urlprefix\url{http://projecteuclid.org/euclid.die/1371586144}

\bibitem{Ch}
X.~Chang, J.~Zhang, J.~Qin, Robustness of nonuniform {$(\mu,\nu)$}-dichotomies
  in {B}anach spaces, J. Math. Anal. Appl. 387~(2) (2012) 582--594.
\newline\urlprefix\url{http://dx.doi.org/10.1016/j.jmaa.2011.09.026}

\bibitem{Ch.La.1999}
C.~Chicone, Y.~Latushkin, Evolution semigroups in dynamical systems and
  differential equations, vol.~70 of Mathematical Surveys and Monographs,
  American Mathematical Society, Providence, RI, 1999.
\newline\urlprefix\url{http://dx.doi.org/10.1090/surv/070}

\bibitem{Da.1972}
R.~Datko, Uniform asymptotic stability of evolutionary processes in a {B}anach
  space, SIAM J. Math. Anal. 3 (1972) 428--445.
\newline\urlprefix\url{http://dx.doi.org/10.1137/0503042}

\bibitem{Ha.2002}
L.~Hatvani, On the asymptotic stability for nonautonomous functional
  differential equations by {L}yapunov functionals, Trans. Amer. Math. Soc.
  354~(9) (2002) 3555--3571.
\newline\urlprefix\url{http://dx.doi.org/10.1090/S0002-9947-02-03029-5}

\bibitem{Li.1934}
T.~Li, Die {S}tabilit\"atsfrage bei {D}ifferenzengleichungen, Acta Math. 63~(1)
  (1934) 99--141.
\newline\urlprefix\url{http://dx.doi.org/10.1007/BF02547352}

\bibitem{Lu.Me.2014}
N.~Lupa, M.~Megan, Exponential dichotomies of evolution operators in {B}anach
  spaces, Monatsh. Math. 174~(2) (2014) 265--284.
\newline\urlprefix\url{http://dx.doi.org/10.1007/s00605-013-0517-y}

\bibitem{Ly.1892}
A.~M. Lyapunov, The general problem of the stability of motion, Internat. J.
  Control 55~(3) (1992) 521--790, translated by A. T. Fuller from {\'E}douard
  Davaux's French translation (1907) of the 1892 Russian original, With an
  editorial (historical introduction) by Fuller, a biography of Lyapunov by V.
  I. Smirnov, and the bibliography of Lyapunov's works collected by J. F.
  Barrett, Lyapunov centenary issue.
\newline\urlprefix\url{http://dx.doi.org/10.1080/00207179208934253}

\bibitem{Ma.1954}
A.~D. Ma{\u\i}zel{\cprime}, On stability of solutions of systems of
  differential equations, Ural. Politehn. Inst. Trudy 51 (1954) 20--50.

\bibitem{Me.1996}
M.~Megan, On {$(h,k)$}-dichotomy of evolution operators in {B}anach spaces,
  Dynam. Systems Appl. 5~(2) (1996) 189--195.

\bibitem{Me.Bu.1993}
M.~Megan, C.~Bu\c se, Dichotomies and {L}yapunov functions in {B}anach spaces,
  Bull. Math. Soc. Sci. Math. Roumanie (N.S.) 37(85)~(3-4) (1993) 103--114.

\bibitem{Mi.Sa.Ku.2003}
Y.~A. Mitropolsky, A.~M. Samoilenko, V.~L. Kulik, Dichotomies and stability in
  nonautonomous linear systems, vol.~14 of Stability and Control: Theory,
  Methods and Applications, Taylor \& Francis, London, 2003.

\bibitem{Na.Pi.1995}
R.~Naulin, M.~Pinto, Roughness of {$(h,k)$}-dichotomies, J. Differential
  Equations 118~(1) (1995) 20--35.
\newline\urlprefix\url{http://dx.doi.org/10.1006/jdeq.1995.1065}

\bibitem{Pe.1930}
O.~Perron, Die {S}tabilit\"atsfrage bei {D}ifferentialgleichungen, Math. Z.
  32~(1) (1930) 703--728.
\newline\urlprefix\url{http://dx.doi.org/10.1007/BF01194662}

\bibitem{Pe.1976}
Y.~Pesin, Families of invariant manifolds that corresponding to nonzero
  characteristic exponents, Izv. Akad. Nauk SSSR Ser. Mat. 40~(6) (1976)
  1332--1379, (Russian) {E}nglish transl. Math. USSR-Izv. 10 (1976),
  1261--1305.
\newline\urlprefix\url{http://dx.doi.org/10.1070/IM1976v010n06ABEH001835}

\bibitem{Pe.1977-1}
Y.~Pesin, Characteristic {L}japunov exponents, and smooth ergodic theory,
  Uspehi Mat. Nauk 32~(4) (1977) 55--112, (Russian) {E}nglish transl. Russ.
  Math. Surv. 32 (1977) 55-114.
\newline\urlprefix\url{http://dx.doi.org/10.1070/RM1977v032n04ABEH001639}

\bibitem{Pe.1977-2}
Y.~Pesin, Geodesic flows in closed {R}iemannian manifolds without focal points,
  Izv. Akad. Nauk SSSR Ser. Mat. 41~(6) (1977) 1252--1288, (Russian) {E}nglish
  transl. Math. USSR-Izv. 11 (1977) 1195--1228.
\newline\urlprefix\url{http://dx.doi.org/doi:10.1070/IM1977v011n06ABEH001766}

\bibitem{Pi.1994}
M.~Pinto, Discrete dichotomies, Comput. Math. Appl. 28~(1-3) (1994) 259--270.
\newline\urlprefix\url{http://dx.doi.org/10.1016/0898-1221(94)00114-6}

\bibitem{Pot.2010}
C.~P{\"o}tzsche, \href{http://dx.doi.org/10.1007/978-3-642-14258-1}{Geometric
  theory of discrete nonautonomous dynamical systems}, Vol. 2002 of Lecture
  Notes in Mathematics, Springer-Verlag, Berlin, 2010.
\newline\urlprefix\url{http://dx.doi.org/10.1007/978-3-642-14258-1}

\bibitem{Pr.Pr.Cr.2012}
C.~Preda, P.~Preda, A.~Craciunescu, A version of a theorem of {R}. {D}atko for
  nonuniform exponential contractions, J. Math. Anal. Appl. 385~(1) (2012)
  572--581.
\newline\urlprefix\url{http://dx.doi.org/10.1016/j.jmaa.2011.06.082}

\bibitem{Pr.Me.1985}
P.~Preda, M.~Megan, Exponential dichotomy of evolutionary processes in {B}anach
  spaces, Czechoslovak Math. J. 35(110)~(2) (1985) 312--323.
\newline\urlprefix\url{http://dml.cz/handle/10338.dmlcz/102019}

\bibitem{Sa.Ba.Sa.2013}
A.~L. Sasu, M.~G. Babu{\c{t}}ia, B.~Sasu, Admissibility and nonuniform
  exponential dichotomy on the half-line, Bull. Sci. Math. 137~(4) (2013)
  466--484.
\newline\urlprefix\url{http://dx.doi.org/10.1016/j.bulsci.2012.11.002}

\bibitem{Sa.2010-1}
B.~Sasu, Integral conditions for exponential dichotomy: a nonlinear approach,
  Bull. Sci. Math. 134~(3) (2010) 235--246.
\newline\urlprefix\url{http://dx.doi.org/10.1016/j.bulsci.2009.06.006}

\end{thebibliography}
\end{document}